\documentclass[11pt]{amsart}
\usepackage{amsmath}
\usepackage{amsthm}
\usepackage{amssymb}
\usepackage{tikz}
\usetikzlibrary{automata,positioning}
\usepackage{enumitem}
\usepackage{mathrsfs}
\usepackage{url}
\usepackage{hyperref}

\newtheorem{theorem}{Theorem}

\newtheorem*{proposition*}{Proposition}

\newtheorem*{lemma*}{Lemma}

\theoremstyle{definition}


\newtheorem*{observation*}{Observation}
\newtheorem*{claim*}{Claim}

\newtheorem{example}[theorem]{Example}
\newtheorem{remark}[theorem]{Remark}

\newtheorem*{remark*}{Remark}

\newtheorem*{conjecture*}{Conjecture}

\newtheorem*{convention*}{Convention}

\renewcommand{\geq}{\geqslant}

\renewcommand{\leq}{\leqslant}

\theoremstyle{plain}

\renewcommand{\tilde}{\widetilde}

\usepackage{mathrsfs}
\usepackage[mathscr]{euscript}

\newcommand{\NN}{\mathbb{N}}
\newcommand{\QQ}{\mathbb{Q}}

\newcommand{\ZZ}{\mathbb{Z}}
\newcommand{\RR}{\mathbb{R}}
\newcommand{\R}{\RR}
\newcommand{\N}{\NN}
\newcommand{\Z}{\ZZ}

\newcommand{\floor}[1]{\left\lfloor #1 \right\rfloor}

\newcommand{\parbreak}[1]{
\begin{center}
***
\end{center}
}

\usepackage{xcolor}
\definecolor{BrickRed}{rgb}{1,0,0}
\definecolor{Black}{cmyk}{0,0,0,1}

\newcommand{\comment}[1]{}%

\usepackage{soul}

\usepackage{stmaryrd}

\begin{document}

\author[J. Byszewski]{Jakub Byszewski$^{2}$}

\email{jakub.byszewski@gmail.com}

\author[J.\ Konieczny]{Jakub Konieczny$^{1,2}$}
\email{jakub.konieczny@gmail.com}

\address[1]{Einstein Institute of Mathematics\\ Edmond J. Safra Campus\\ The Hebrew University of Jerusalem\\ Givat Ram\\ Jerusalem, 9190401\\ Israel}

\address[2]{Department of Mathematics and Computer Science\\
Institute of Mathematics\\
Jagiellonian University\\
ul. prof. Stanis\l{}awa \L{}ojasiewicza 6\\
30-348 Krak\'{o}w\\
Poland}

\title{Factors of generalised polynomials and automatic sequences}

\maketitle 

\begin{abstract} The aim of this short note is to generalise the result of Rampersad--Shallit saying that an automatic sequence and a Sturmian sequence cannot have arbitrarily long common factors. We show that the same result holds if a Sturmian sequence is replaced by an arbitrary sequence whose terms are given by a generalised polynomial (i.e., an expression involving algebraic operations and the floor function) that is not periodic except for a set of density zero.
	 
\end{abstract}

A Sturmian sequence is defined as an infinite word with values $0$ and $1$ that encodes the set of times at which the orbit of a point with respect to an irrational rotation by $\theta$ hits a given arc of length $\theta$. It is well-known that a Sturmian sequence is not automatic, i.e., it cannot be produced by a finite automaton that reads the base-$k$ digits of the input in some fixed base $k\geq 2$.  In a recent note \cite{RampersadShallit2018}, Rampersad and Shallit have shown that not only is it impossible for automatic and Sturmian words to coincide---their common factors (i.e., finite blocks of consecutive symbols) have in fact little in common.

\begin{theorem}[Rampersad--Shallit]\label{thmRS} Let $x$ be a $k$-automatic sequence and let $a$ be a Sturmian sequence. There exists a constant $C$ (depending on $x$ and $a$) such that if $x$ and $a$ have a factor in common of length
$n$ then $n \leq  C$. \end{theorem} 

A Sturmian sequence can be equivalently defined as an infinite word $a_0 a_1 a_2 \cdots$ of the form $$a_n = \floor{ \alpha (n+1) + \rho } - \floor{ \alpha n + \rho } - {\floor{\alpha}},$$ where $\alpha,\rho \in \RR$ with $\alpha \not \in \QQ$.  
An expression of this form is a very simple example of a \emph{generalised polynomial}, i.e., a function $a\colon \N_0 \to \R$ given by an expression involving real constants, the algebraic operations of addition and multiplication along with the (possibly iterated) use of the floor function. By a result of Bergelson--Leibman \cite{BergelsonLeibman2007}, generalised polynomials are intimately related to dynamics on nilmanifolds. (In the case of a Sturmian sequence the corresponding nilmanifold is the circle with the irrational rotation by   $\theta$.) In a recent work, we have shown that generalised polynomials cannot be automatic unless they are periodic outside of a set of density zero (for this and related results see \cite{ByszewskiKonieczny2017} and \cite{ByszewskiKonieczny2016}). It is therefore natural to ask whether the result on common factors generalises to this more general context. This is indeed the case. The following result generalises \cite[Theorem B]{ByszewskiKonieczny2017}. We remind the reader that a set $A\subset \N_0$ has upper Banach density zero if for every $\varepsilon>0$ any interval of sufficiently large length $l$ contains at most $\varepsilon l$ elements of $A$. Note that whether two sequences have arbitrarily long common factors or not is independent of changing their values on a set of upper Banach density zero.

\begin{theorem}\label{main} Let $x$ be a k-automatic sequence and let $a$ be a generalised polynomial sequence. Suppose that $x$ and $a$ share arbitrarily long factors. Then there exists a periodic sequence $p$ such that $a$ and $p$ coincide outside of a set of upper Banach density zero and $x$ has arbitrarily long common factors with $p$. \end{theorem} 

Since a Sturmian word has irrational letter frequencies, it is not periodic even up to density zero, and hence it cannot have arbitrarily long common factors with an automatic word. Thus Theorem \ref{main} generalises Theorem \ref{thmRS}.

A proof of Theorem \ref{main} follows closely the lines of the proof of the Sturmian case by Rampersad--Shallit. However, while in the case of Sturmian sequences one may reason directly using Kronecker's theorem, we need to use the Bergelson--Leibman theorem to represent a generalised polynomial in terms of values of a semialgebraic function along an orbit on  a nilmanifold. These methods are applied in an analogous manner as in the proof of \cite[Theorem B]{ByszewskiKonieczny2017}.

We also note that the result continues to hold if the class of automatic sequences is replaced by a slightly more general class of weakly periodic sequences defined in \cite{ByszewskiKonieczny2017}. In order not to clutter the proof, we first show the claim for automatic sequences, and then in Remark \ref{resultforwp} show how to modify the argument in order to get the stronger claim. 

\begin{proof}[Proof of Theorem \ref{main}] Let $x$ be an automatic sequence and let $a\colon \N_0 \to \R$ be a generalised polynomial. Assume that $x$ and $a$ have arbitrarily long common factors. We first show that we may assume that $a$ takes only finitely many values. 

The automatic sequence $x$ takes values from a finite set $\Omega$. Let $\varpi\in \R$ be an element outside of $\Omega$. It is easy to see (using  \cite[Lemma 1.2]{ByszewskiKonieczny2016}) that the function $\tilde{a} \colon \N_0 \to \R$ given by $$\tilde{a}(n)=\begin{cases} a(n) &\text{if } a(n)\in \Omega,\\ \varpi &\text{otherwise}\end{cases}$$ is also a generalised polynomial. Clearly, $x$ and $a$ have the same common factors as $x$ and $\tilde{a}$. Thus, if we knew the claim for the sequences $x$ and $\tilde{a}$, we could conclude that $\tilde{a}$ takes the value $\varpi$ on a set of upper Banach density zero, and hence $a$ and $\tilde{a}$ coincide outside of a set of upper Banach density zero. Therefore, without loss of generality we may assume that $a$ takes only finitely many values $\{c_1,\ldots,c_r\}$. Furthermore, a similar argument that involves merging all but one of the values $c_1,\ldots,c_r$ allows us to assume that $x$ and $a$ admit only two values that may further be taken to be $0$ and $1$.

We now apply the Bergelson--Leibman theorem \cite{BergelsonLeibman2007} to find a representation of $a$ in terms of an orbit on a nilmanifold. (For more information on nilmanifolds and generalised polynomials, see \cite{BergelsonLeibman2007} and references therein; the necessary prerequisites are also briefly discussed in \cite[pp.\ 9--10]{ByszewskiKonieczny2017} and we will use this reference for convenience.) By \cite[Theorem 1.13]{ByszewskiKonieczny2017}, there exists a minimal nilsystem $(X,T)$, a point $z\in X$, and a semialgebraic subset $A$ of $X$ such that
\begin{equation}\label{reprezentacja} a(n) = \begin{cases}  1 &\text{ if } T^n(z) \in A,\\ 0 &\text{ otherwise.}\end{cases}\end{equation}

We will first prove the result under the additional assumption that $(X,T)$ is totally minimal. The sequence $x$ is $k$-automatic, and we can choose an integer $r$ such that the cardinality of the $k$-kernel of $x$ is strictly smaller than $k^r$. Then there exist integers $0\leq s <t <k^r$ such that $x(k^r n+s)=x(k^r n+t)$ for $n\in \N_0$. Let $d=t-s$. Since $a$ and $x$ have arbitrarily long common factors, we can conclude that for each $L$ there exist $n_0 \geq 0$ and $0\leq u < k^r$ such that $$a(n) = a(n + d) \text{ for } n_0 \leq n < n_0 + L \text{ and } n \equiv u \pmod{k^r}.$$

We claim that it follows that $T^d(\mathrm{int}(A)) \subset \mathrm{cl}(A)$. To this end, pick a point $y \in \mathrm{int}(A)$ and pick a neighbourhood $U \subset A$ of $y$ contained in A. Since $(X,T)$ is totally minimal, we may apply well-distribution of generalised polynomials \cite[Theorem B]{BergelsonLeibman2007} to conclude that for large enough $L$ and all $n_0\geq 0$ there exists an integer $n$ such that $n_0 \leq n < n_0 + L$, $ n \equiv u \pmod{k^r}$, and $T^n(z)\in U$. It follows that $a(n) = a(n+d) =1$, and hence $T^d(U) \cap A \neq \emptyset$. Since $U$ was an arbitrary neighbourhood of $y$, we conclude that $T^d(y) \in \mathrm{cl}(A)$. It follows that $T^d(\mathrm{int}(A)) \subset \mathrm{cl}(A)$. Since $A$ is semialgebraic, the sets $\mathrm{int}(A)$ and $\mathrm{cl}(A)$ have the same measure (this follows, e.g., from \cite[Proposition 2.8.13]{BCR}). Since minimal nilsystems are ergodic, we conclude that $A$ is either a set of measure zero or of full measure, and hence $a$ is in any case constant outside of a set $Z$ of density zero. By \cite[Corollary 1.12]{ByszewskiKonieczny2017}, this exceptional set $Z$ is even of upper Banach density zero. Therefore, since $x$ has arbitrarily long common factors with $a$, it follows that $x$ is constant on arbitrarily long intervals. This concludes the proof of the theorem in the totally minimal case.

In order to reduce the claim to the totally minimal case, we apply \cite[Remark 1.14]{ByszewskiKonieczny2017}. We get that there exists a positive integer $i$ such that for all integers $0\leq j \leq i-1$ the generalised polynomial $a_{i,j}(n) = a(in + j)$ has a representation of the form \eqref{reprezentacja} with $(X, T)$ totally minimal. If $a$ has arbitrarily long common factors with $x$, then each $a_{i,j}$ has arbitrarily long common factors with one of the automatic sequences $x_{i,j'}(n) = x(in + j')$ for some $0\leq j'\leq i-1$. By the totally minimal case, we may conclude that $a$ coincides with a periodic sequence $p$ outside of a set of upper Banach density zero. Since $x$ has arbitrarily long common factors with $a$, it follows that $x$ coincides with $p$ on arbitrarily long intervals.
\end{proof}

\begin{remark}\label{resultforwp} We now comment on how to modify the proof of Theorem \ref{main} in order to generalise the statement to weakly periodic sequences. We remind the reader that a sequence $x\colon \N_0 \to \R$ is called weakly periodic if for any sequence $x'(n)=x(an+b)$ obtained from $x$ by restriction to an arithmetic subsequence with $a\in \N, b\in \N_0$ there exist $q\in \N$, $r,r'\in \N_0$ with $r \neq r'$ such that $x'(qn+r)=x'(qn+r')$. 

If $x$ is a weakly periodic sequence that attains only  finitely many values, then the proof is mutatis mutandis the same as above. However, in contrast to automatic sequences, weakly periodic sequences may well attain infinitely many values, in which case one needs to reason more carefully. The first obstacle is that the Bergelson--Leibman theorem applies only to bounded sequences.

Let $x$ be a weakly periodic sequence and let $a$ be a generalised polynomial sequence that have arbitrarily long common factors. We will prove that there exist a periodic sequence $p$ such that $a$ and $p$ coincide outside of a set of upper Banach density zero and $x$ has arbitrarily long common factors with $p$. Since $a$ and $x$ take only countably many values, there exists $\theta \in \R$ such that for any two distinct elements $s,s'\in S:=\{a(n) \mid n\in \N_0\}\cup\{x(n)\mid n\in \N_0\}$, we have $\theta(s-s')\notin\Z$. Let $\{t\}=t-\lfloor t \rfloor$ denote the fractional part of $t\in \R$. Replacing $x(n)$ with $\{\theta x(n)\}$ and $a(n)$ with $\{\theta a(n)\}$ allows us to assume that $a$ and $x$ are both bounded.

Applying the Bergelson--Leibman theorem, we get a representation of $a$ in the form $a(n)=f(T^n z)$, where $(X,T)$ is a minimal nilsystem, $z\in X$, and $f\colon X \to \R$ is a semialgebraic function. The same reasoning as above allows us to assume that $(X,T)$ is totally minimal. We now claim that $f$ is constant on a set of full measure. Assume the contrary. Then there exists an open semialgebraic subset $U \subset X$ on which $f$ is continuous and takes at least two distinct values $t_1<t_2$. Applying the above reasoning to the characteristic sequences of the sets $\{n\in \N_0 \mid a(n)<(t_1+t_2)/2\}$ and $\{n\in \N_0 \mid x(n)<(t_1+t_2)/2\}$ (which take values $0$ and $1$), we obtain a contradiction. Thus $f$ is constant on a set of full measure, and hence $x$ is constant outside of a set of upper Banach density zero. The claim follows. \end{remark}

\begin{example}
	Consider the Thue--Morse sequence $t(n) = s_2(n) \bmod{2}$ where $s_2(n)$ denotes the sum of binary digits, and the generalised polynomial
	$$
		a(n) = 
		\begin{cases}
			1 &\text{if } \sqrt 2 n^2 \in (-1/4,1/4) \bmod{1};\\
			0 &\text{otherwise.} 
		\end{cases}
	$$ 
Theorem \ref{main} guarantees that the common factors of $t$ and $a$ have bounded length. In fact, their common factors are precisely: $001011010$, $010110100$, $110100101$, $101001011$, $10100110010$, $01001100101$, $01100101101$, $01001011001$, $01011001101$, $10110011010$, $10011010010$, $10110100110$ and their subfactors. In particular, the longest common factor has length $11$. We explain below how such a computation can be made.

In principle, one can adapt the argument in Theorem \ref{main} to give explicit bounds on the length of the common factors that $t$ and $a$ share. In practice, however, such bounds would be both poor and difficult to obtain, so we proceed otherwise. 

Note first that for a given word $w = (w(0),\dots,w(l-1)) \in \{0,1\}^l$ it is computationally feasible to verify whether $w$ is a factor of $a$. Because $\sqrt{2}$ is irrational, we may freely replace the open interval in the definition of $a$ with a closed one. Hence, if $w$ is a factor of $a$ (appearing at the position $n_0$) then there exists $\beta,\gamma \in [0,1)$ (given by $\beta = \{2 \sqrt{2} n_0\}$ and $\gamma = \{\sqrt{2} n_0^2\}$) such that for $0 \leq m < l$ it holds that 
\begin{align}
	\sqrt{2} m^2 + \beta m + \gamma \in (-1/4,1/4) \bmod{1} \text{ if } w(m) = 1,\label{eq:432a}\\
	\sqrt{2} m^2 + \beta m + \gamma \not \in [-1/4,1/4] \bmod{1} \text{ if } w(m) = 0.\label{eq:432b}
\end{align}
Conversely, the set of $\beta, \gamma$ satisfying \eqref{eq:432a}--\eqref{eq:432b} is open, so it follows from Weyl equidistribution theorem that if \eqref{eq:432a}--\eqref{eq:432b} hold for some $\beta,\gamma$ then $w$ is a factor of $a$. For a fixed choice of $w$, the condition \eqref{eq:432a}--\eqref{eq:432b} is simply a system of linear inequalities\footnote{Strictly speaking, because of the appearance of the fractional part, the situation is marginally more complicated. However, there is a finite number of values that $\floor{\sqrt{2} m^2 + \beta m + \gamma}$ can take for $0 \leq m < l$ when $\beta, \gamma \in [0,1)$. 
Once the values $\floor{\sqrt{2} m^2 + \beta m + \gamma}$ for $0 \leq m < l$ are fixed, \eqref{eq:432a}--\eqref{eq:432b} are precisely equivalent to an alternative of finitely many systems of linear inequalities.}
in real parameters $\beta,\gamma$, and it is computationally feasible to verify whether it is solvable or not. 

The task of finding the factors of $t$ is considerably simpler. Indeed, it follows from \cite[Example 10.9.3]{AlloucheShallit-book} that each factor of $t$ of length $\leqslant 2^l+1$ for $l\geq 1$ can be found at a starting position $\leqslant 9\cdot2^{l-1}$. In particular, all factors of $t$ of length $\leq 17$ can be found by examining the first $88$ entries.

In order to identify all common factors of $t$ and $a$, we verified for each factor of $t$ of length $< 13$ whether it is a factor of $a$ using the argument described above\footnote{Computations were performed using Wolfram Mathematica; source code available from the authors.}. Since no factor of length $12$ was found, these are the only shared factors. 

Similar arguments can be applied to other generalised polynomials and automatic sequences in place of $a$ and $t$.	
\end{example}

\subsection*{Acknowledgements}
We thank Narad Rampersad for his useful comments. JB is supported by the National Science Centre, Poland (NCN) under grant no.\ DEC-2012/07/E/ST1/00185. JK  is supported by the European Research Council (ERC) under grant ErgComNum 682150.

\bibliographystyle{alpha}
\bibliography{bibliography}

\end{document}